\documentclass[11pt]{article}
\usepackage{amsmath,amsthm,
amsfonts,amssymb,mathrsfs,epsfig,bm}
\usepackage[usenames]{color}
\usepackage[colorlinks=true]{hyperref}
\usepackage{graphicx}

\parskip        \medskipamount
\newtheorem{theorem}{Theorem}
\newtheorem{lemma}{Lemma}
\newtheorem{corollary}{Corollary}
\newtheorem{proposition}{Proposition}

\newtheorem{remark}{Remark}
\theoremstyle{remark}

\def\Q{\mathbb{Q}}
\def\R{\mathbb{R}}

\def\P{\mathbb{P}}

\renewcommand{\phi}{\varphi}
\renewcommand{\epsilon}{\varepsilon}

\newcommand{\1}{{\text{\Large $\mathfrak 1$}}}

\renewcommand{\limsup}{\varlimsup}
\renewcommand{\liminf}{\varliminf}

\definecolor{mygray}{gray}{0.9}
\definecolor{deeppink}{RGB}{255,20,147}
\definecolor{mygreen}{rgb}{0.05, 0.576, 0.03}
\definecolor{myred}{rgb}{0.768, 0.09, 0.09}

\long\def\symbolfootnote[#1]#2{\begingroup
\def\thefootnote{\fnsymbol{footnote}}\footnote[#1]{#2}\endgroup}

\usepackage{undertilde}
\usepackage{MnSymbol}

\usepackage{authblk}
\begin{document}

\title{A generalization of Kingman's model of selection and mutation and the Lenski experiment}
\author[]{Linglong Yuan\footnote{Linglong Yuan gratefully acknowledges support by DFG priority programme SPP 1590 {\it Probabilistic structures
in evolution}}}
\affil[]{Institute of Mathematics\\Johannes Gutenberg University Mainz}
\date{\today}
\maketitle

\begin{abstract}
Kingman's model of selection and mutation studies the limit 
type value distribution in an asexual population
of discrete generations and infinite size undergoing 
selection and mutation. This paper generalizes the model to analyse the long-term evolution of 
Escherichia. coli in Lenski experiment. Weak assumptions for fitness functions 
are proposed and the mutation mechanism is the same as in Kingman's model.
General macroscopic epistasis are designable through
fitness functions.  Convergence to the unique limit type distribution is obtained.  
\\
\\

\noindent \textit{Keywords:} Selection-Mutation balance, House of Cards, Type distribution, Fitness function, Experimental evolution, Lenski experiment, Escherichia. coli, Condensation.\\
\\
\textit{MSC Subject classification:} 92D10 (primary), 37N25, 60F05, 92D15 (secondary)

\end{abstract}

\section{Introduction}
Evolutionary forces in a population vary from macroscopic scale to microscopic scalebincluding random environment, migration, natural selection,  macroscopic epistasis (or individual interaction),
microscopic epistasis,  
 and linkage and dominance, clonal interference, mutation, genetic drift, 
recombination, and so on. 
Recently, mathematicians are interested in incorporating as many factors
as possible in an evolutionary model, either deterministic or stochastic,
 to understand the contribution of each factor
and to see which state the model can reach in the limit (see for example 
\cite{Age, NFM1, NFM2} among numerous works). However one would expect a high level of
 complexity of modelling and analysis when many factors enter into play. 

Kingman \cite{K78} suggested that one can regard an equilibrium of the evolutionary model
as existing because of 
two preponderant factors, other phenomena causing perturbations of the equilibrium. 
The pair of factors in his model are selection and mutation. This particular case
had also been the subject of study of Moran (\cite{1M76, 2M76,M77}) almost at the same time.

More specifically, Kingman \cite{K78} proposed a one-locus, discrete generation 
model under selection and mutation with an infinite number of possible 
alleles which have continuous effects on a quantitative type. 
The continuum-of-alleles models were introduced by 
Crow and Kimura \cite{CK64} and Kimura \cite{K65} and are used frequently in
quantitative genetics, since types usually have a polygenetic basis.

Kingman's idea can be applied to model the Lenski experiment which 
investigates the long term evolution of E.coli in the laboratory. 
Indeed, the application goes to various evolutionary models and one major parameter is 
how selection influences the population. That generates many variants of 
Kingman's model and a general treatment is required.

The paper aims to establish a general model which covers the Kingman's setting and 
can be applied to Lenski experiment. In section 2, we show briefly the Kingman's model and the main observations in
Lenski experiment. This section is the motivation of the paper, but the reader can skip it for the first reading since
we will come back for applications. 
In section 3, we introduce a general setting with 3 assumptions on the fitness function. We give the main results for the general model when some or all assumptions hold. Section 4 is devoted to  
proofs and in section 5, we show the applications to Kingman's model and Lenski experiment. Section 6 
summarizes the main contribution of the paper and discusses the comparison of our model with other works, especially with \cite{GKWY} on Lenski experiment.   

\section{Kingman's model and Lenski experiment}
\subsection{Kingman's model}
  The model considers an effectively infinite population that reproduces 
 asexually and has discrete generations. 
 It studies a specific type and the selection influences the population through the fitness which is the 
   offspring size and  depends (possibly not only) on its type value $x$,
 a real number in the space $\mathcal{M}:=[0,M]\subset \R^+_0$ where $M$ is a positive real number.
   Let $\P(\mathcal{M})$ denote the set of probability measures on $\mathcal{M}$. For any $u\in\P(\mathcal{M})$,  let $m_{u}$ denote the upper limit of the support of $u$, i.e., $ m_{u}=\sup\{x: x\in\mathcal{M}, u([x-\epsilon,x])>0, \forall 0<\epsilon<x\}$. So $m_{u}$ is the largest type value an individual can take in a population with type distribution $u$. 

 
 Assume that 
 each individual per generation 
 mutates independently with probability $\beta$ ($0<\beta<1$) and the mutant type distribution is 
 the probability measure $q$ on $\mathcal{M}$, independent of parent's type. Kingman \cite{K78,K77} argued that the tendency for most mutations to be deleterious  might be reflected in a model in which the gene after mutation is independent of that before, the mutation having destroyed the biochemical ``house of card" built up by evolution. 
 
 The fitness function in this model is $x\mapsto x, x\in\mathcal{M}.$
Let $(p_i)_{i\geq 0}$ denote the sequence of type distributions of generations $i$ on $\mathcal{M}$ with $p_0$
given as a parameter;
Then $(p_i)_{i\geq 0}$ are defined recursively: 
 \begin{equation}\label{king}p_i(dx)=(1-\beta)\frac{xp_{i-1}(dx)}{\int xp_{i-1}(dx)}+\beta q(dx), i\geq 1. \end{equation}
 
 { In particular, we set $p_i(dx)=(1-\beta)p_{i-1}(dx)+\beta q(dx)$, if $p_{i-1}=\delta_0$, the Dirac measure at $0$.}


\begin{remark}
Due to the expression of (\ref{king}), it is clear that $m_{p_i}\leq \max\{m_{p_0},m_q\}$ for any $i\geq 0$. So letting 
$M=\max\{m_{p_0},m_q\}$ and ${\mathcal{M}}=[0,\max\{m_{p_0},m_q\}]$ does not change any $p_i$. Since $m_q\leq m_{p_1}$, one can assume $m_q\leq m_{p_0}$, otherwise we take $p_1$ as $p_0$. 
For convenience, we introduce:

Convention $(*)$: 
$$m_q\leq m_{p_0}, M=m_{p_0} \text{ and }{\mathcal{M}}=[0, m_{p_0}].$$
\end{remark}
If a sequence of measures (not necessarily probability measures) $(h_i)_{i\geq 0}$ converges in total variation sense to a measure $h$, that is, the total variation of $h_i-h$ tending to $0$, then for abbreviation, we say $(h_i)_{i\geq 0}$ converges strongly to $h$. 

Kingman specifically takes $M=1$ in his model. Based on the value of $\int\frac{q(dx)}{1-x}$, Kingman (\cite{K78}) proved that:\\

\begin{theorem}[Kingman]\label{King}
\textit{Case $1$}:
 $\int \frac{q(dx)}{1-x}>\beta^{-1}.$  Then $(p_i)_{i\geq 0}$ converges strongly to 
$$p^*(dx)=\frac{\beta sq(dx)}{s-(1-\beta)x},$$
with $s$ being the unique solution of $\int\frac{\beta sq(dx)}{s-(1-\beta)x}=1$.\\
\textit{Case 2}: $\int \frac{q(dx)}{1-x}\leq \beta^{-1}$.
Then $(p_i)_{i\geq 0}$ converges weakly to 
$$p^*(dx)=\frac{\beta q(dx)}{1-x}+(1-\int\frac{\beta q(dy)}{1-y})\delta_{1}(dx),$$ 
here $\delta_1(dx)$ is the Dirac measure at $1.$ 
\end{theorem}

Therefore the sequence $(p_i)_{i\geq 0}$ converges at least in the weak sense to a limit distribution $p^*$ which depends only on $q$ and $\beta$, regardless of the specific form of $p_0$.  Biologically, it can be seen as a stability property of the population. 

Next we introduce the Lenski experiment and use an iteration similar to (\ref{King}) to model the evolution of E.coli.
\subsection{Lenski experiment and modelling}

The Lenski experiment is a long-term evolution experiment with E.coli, founded by Richard E. Lenski in 1988 in the laboratory. The experiment is decomposed into {\it daily cycles}. Every day starts by sampling approximately $5\cdot 10^6$ bacteria from those available in the medium that was used the previous day. This sample is then transferred to a new glucose-limited minimal medium and reproduce (asexually) until the medium is deployed, i.e., when there is no more glucose available. Around $5\cdot 10^8$ cells are present at the end of each day. So the size grows by approximatively $100$ times from the beginning of a day to the end and a sample of percentage around $1\%$ will be chosen for the next day. The closely $30$-year ongoing experiment has run more than $60,000$ generations. We refer to \cite{GKWY} for a more detailed presentation and references therein. 

There are $12$ populations founded from a common ancestor. Samples, called by Lenski ``fossil record'', are frozen every 500 generations. Once bacterium is frozen, we consider it stopping biological activities inside the body, which is how  the name ``fossil record'' makes sense. The records are regarded as stocked information of evolutionary trajectories of populations. 

They define the fitness as the dimensionless ratio of the competitors' realized reproduction rates. Basically, we let two populations of the same number of individuals, one of the founder ancestors and one of evolved strain, to be together in a medium at the beginning of a day. The fitness of the evolved strain is the ratio of the (exponential) reproduction rate of the strain observed at the end of the day and the reproduction rate of the ancestor strain. So in this definition, fitness is a {\it relative} quantity that measures the reproduction rate of the whole population. However mathematically, one can directly model the natural (non-relative) reproduction rate of each bacterium.  To unify notations, we shall consider the natural reproduction rate as the type and our fitness, different from that of Lenski, is the offspring size in the next generation. 

Wiser et al \cite{W13} showed that the (relative) reproduction rate increases but decelerates.  They compared the hyperbolic and sublinear power law increasing models, the former having a bound and the latter none.  It turns out that the hyperbolic model fits to the first $10,000$ generations, but for a long term about $50,000$ generations, the sublinear power law model is more significant. 

The unboundedness of the sublinear power law curve can be explained by the fact that  highly beneficial mutations happen rarely but consistently and with probability $1$ some of them fixate the population, although after a probably very long time. Therefore, on a very long period of time, one can consider that there is a bound (or even a pattern) for the reproduction rates of new mutants. 

More specifically, let $p_0$ be the initial type distribution (or reproduction rate distribution) and $q$ the mutant type distribution such that $0\leq m_q\leq m_{p_0}<\infty$. Let $M,\mathcal{M}$ be defined from $p_0, q$ by convention $(*)$. 
Let the population grow exponentially according to the reproduction rate of each individual until the total amount 
reaches the capacity $\gamma (\gamma>1)$ (assuming the initial amount $1$). In Lenski experiment, $\gamma\approx 100$. We then sample $1/\gamma$ proportion of the population at the end of day $0$ to constitute the population at the beginning of day $1$. However, we have new mutants arriving along the whole day. To combine the mutation and selection together, we assume that a $\beta (0<\beta<1)$ proportion of the sample consists of mutation population with type distribution  $q$ and the rest $1-\beta$ proportion stays unchanged.  For any $u\in\P(\mathcal{M})$, let $t_{u}$ be the unique solution of $\int e^{t_ux}u(dx)=\gamma$. The type distribution $p_i$ at generation (the beginning of day) $i$ is determined as follows: 
\begin{equation}\label{lenskim}
\text{\bf Lenski model :\,\,\,\,\,\,\,\,}p_i(dx)=(1-\beta)\frac{e^{t_{p_{i-1}}x}}{\gamma}p_{i-1}(dx)+\beta q(dx),\,i\geq 1, x\in \mathcal{M}.\end{equation}
Given type distribution $p_{i-1}$, $t_{p_{i-1}}$ is the time the population needs to reach capacity $\gamma$ within one day. 
The fitness function is given by $x\mapsto e^{t_{p_{i-1}}x},x\in \mathcal M$.  A division by $\gamma$ represents the uniform sampling of proportion $1/\gamma$ from the population at the end of the day. 

It turns out that Kingman's method cannot be simply applied to (\ref{lenskim}). To find a solution, we will establish a more general model which comprises (\ref{King}) and (\ref{lenskim}) and find a general solution using a different approach. 


\section{General model and main results}
\subsection{Modelling}

Let $p_0$ be the initial type distribution and $q$ the mutant type distribution such that $0\leq m_q\leq m_{p_0}<\infty$. Let $M,\mathcal{M}$ be defined from $p_0,q$ by convention $(*)$. Let $\beta (0<\beta<1)$ be the probability of mutation per capita per generation.  In the rest of the paper, $M, \mathcal{M}, \beta$ are fixed. 

Let $w(x,u)$ be the (nonnegative) fitness function of an individual with type value $x\in\mathcal{M}$ and type distribution $u\in\P(\mathcal{M})$. Assume that given $u$, $w(x,u)$ is a measurable function on $\mathcal{M}.$ 
Then the type distribution $p_i$ is defined via a recurrence relation:
 \begin{equation}\label{general}p_i(dx)=(1-\beta)\frac{w(x,p_{i-1})p_{i-1}(dx)}{\int w(x,p_{i-1})p_{i-1}(dx)}+\beta q(dx), i\geq 1. \end{equation}
In particular, if $\int w(x,p_{i-1})p_{i-1}(dx)=0$ for some $i$, then 
$$p_i(dx)=(1-\beta)p_{i-1}(dx)+\beta q(dx).$$
\begin{remark}
Note that the fitness $w(x,p_{i-1})$ depends on $p_{i-1}$ only in Lenski model, not in Kingman's model.
\end{remark}
We introduce a new notion, namely the ``selective advantage'' of an individual with type value $x$ at generation $i-1$:

\begin{equation}\label{sa}s(x,p_{i-1}):=\frac{w(x,p_{i-1})}{\int w(x,p_{i-1})p_{i-1}(dx)}.\end{equation}
So (\ref{general}) can be written as 

\begin{equation}\label{selad}
p_i(dx)=(1-\beta)s(x,p_{i-1})p_{i-1}(dx)+\beta q(dx), i\geq 1.
\end{equation} 
One can see that $x$ is fitter than $y$ if and only if $s(x,p_{i-1})\geq s(y,p_{i-1})$. 

We would like to achieve that $(p_i)_{i\geq 0}$ converges to a limit distribution given any parameters $p_0, q$. To this end, it is not surprising that one needs some assumptions on $w$. The in-total 3 assumptions will be introduced and discussed in the next subsection where main results will also be announced.

\subsection{Main results}

\subsubsection{Assumption 1}

We start with some notations.  For any $p\in\P(\mathcal{M})$, the distribution function $D_p$ is defined as:$D_{p}(x)=p([0,x])$ for any $x\in M$. For any $u,v\in \P(\mathcal{M})$, 
 say $u$ is \textit{stochastically dominated} by $v$, denoted by $u\leq v$, if and only if $D_{u}(x)\geq D_{v}(x)$ for any $x\in \mathcal{M}$. 
 
Let $\Q(\mathcal{M})$ be the set of all (non-negative) finite measures on $\mathcal{M}$. Let $u\in\Q(\mathcal{M})$ and a measurable set $A\subset \mathcal{M}$, 
then $u_A$ is defined as $u_A(B):=u(B\cap A)$
for any measurable set $B\subset \mathcal{M}$. Let $u,v\in\Q(\mathcal{M})$ and a measurable set $A\subset \mathcal{M}$. Say $u$ is \textit{a component} of $v$ on $A$ (or $u_A$ is a {\it component} of $v_A$), denoted by $u_A\prec v_A$, if and only if  $v_A-u_A\in\Q(\mathcal{M})$. 

 {\it Assumption $1$: For any $u,v\in\P(\mathcal{M})$, if $v_{[0,M)}\prec u_{[0,M)}$, then 
\begin{equation}\label{dominance}s(x,u)\geq s(x,v),\,\,\forall x\in [0,M].\end{equation}}
\noindent{\it Biological interpretation:} Note that $v$ is supported relatively on larger values  than $u$. In Kingman's model, fitter types are represented by larger numbers. In this spirit, a population with distribution $v$ is fitter than the one with $u$.  So given an individual with any type value $x$, it should have less selective advantage in $v$ than in $u$, which justifies the assumption (\ref{dominance}).

For some $p_0,q$, the limit theorem requires only the assumption 1: 

\begin{theorem}\label{c1}Under assumption 1, if $p_0=\delta_M$, then  $(p_i)_{i\geq 0}$ converges strongly to a probability measure $p^*\in\P(\mathcal{M})$ and $p^*$ depends only on $M,q,\beta$.
 Moreover $(p_i)_{[0,M)}\prec (p_{i+1})_{[0,M)}\prec(p^*)_{[0,M)}$ for any $i\geq 0$. 
\end{theorem}
This specific $p^*$ will turn out to be the unique limit distribution in the general model. 

\begin{proposition}\label{end}
Under assumption 1, for any $p_0,q$, if $p_i(\{M\})$ converges to $p^*(\{M\})$, then $(p_i)_{i\geq 0}$ converges strongly to $p^*$.   
\end{proposition}

\begin{corollary}\label{nomass}
Under assumption 1, for any $p_0,q$, if $p^*(\{M\})=0$, then $(p_i)_{i\geq 0}$ converges strongly to $p^*$.   
\end{corollary}

\begin{theorem}\label{c11}Under assumption 1, if $q(\{M\})>0$, then  $(p_i)_{i\geq 0}$ converges strongly to $p^*$. 
\end{theorem}
The following corollary gives a practical criterion for strong convergence.   
\begin{corollary}\label{forking}
Under assumption 1, for any $p_0,q$, if there exists a measurable function $f(x)$ on $\mathcal{M}$ such that $p^*(dx)=f(x)q(dx)$. Then $(p_i)_{i\geq 0}$ converges strongly to $p^*$.
\end{corollary}
The proof is straightforward: if $p^*(\{M\})=0$, then Corollary \ref{nomass} applies. If $p^*(\{M\})>0$ then $q(\{M\})>0$. In this case Theorem \ref{c11} applies. 

\subsubsection{Assumption 2}

{\it Assumption $2$: For any $u,v\in\P(\mathcal{M})$ and $0<\epsilon<1$, if $v_{[0,M)}\prec u_{[0,M)}$ and there exists $0\leq a<M$ such that $D_u(a)\geq D_v(a)+\epsilon$, then there will be a number $c(a,\epsilon)>1$ such that  
\begin{equation}\label{ratio}
s(M,u)\geq s(M,v)c(a,\epsilon).
\end{equation}}
\noindent{\it Biological interpretation:} Distribution $u$ puts more weight on $[0,M)$ and especially $\varepsilon$ more on $[0,a]$. So a population with type distribution $u$ is substantially less fitter than the one with $v$. This substantiality is explained by $c(a,\epsilon)>1$ at point $M$.

Based on assumptions 1 and 2, one can prove limit theorem for $0<p_0(\{M\})<1, q(\{M\})=0$. 





\begin{theorem}\label{c2}
Under assumptions 1, 2, if $0<p_0(\{M\})<1,  q(\{M\})=0$, then for any $0\leq a< M$, $(p_i)_{[0,a]}$ converges strongly to $(p^*)_{[0,a]}$. As a consequence, $(p_i)_{i\geq 0}$ converges weakly to $p^*$. 
\end{theorem}

\subsubsection{Assumption 3}
Now one can proceed to assumption 3 which will solve the problem for the last situation: $p_0(\{M\})=q(\{M\})=0.$ A few notations are needed.  For any $h\in\P(\mathcal{M})$, let $h^a=h_{[0,a)}+h([a,m_h])\delta_a$ for any $0\leq a\leq m_h$ and $h^a=h$ for $a>m_h$. Given $q\in\P(\mathcal{M})$ and $0\leq a\leq M$, let $p_0=\delta_{a}$ be the initial type distribution and $q^a$ the mutation type distribution, with $p_i$ the type distribution of generation $i$.   Note that  $m_{q^a}\leq m_{p_0}=a. $ By Theorem \ref{c1}, $(p_i)_{i\geq 0}$ converges strongly to a limit probability measure, denoted by $p^{a,*}$. Remark that $p^{M,*}=p^*$. The assumption $3$ is the following: 

{\it Assumption 3: In the above framework,  $p^{a,*}$ converges weakly to $p^*$ as $a$ tends to $M$ from left.}

\noindent{\it Biological interpretation:} This assumption says that the population is healthy and stable, in the sense that, by a slight type value decrease among the fittest individuals both in $p_0$ and/or $q$, the limit type distribution should not be changed too much. This assumption shows the stability of population resistant to type value perturbance. 

\begin{theorem}\label{c3}
Under assumptions 1,2 and 3, if $p_0(\{M\})=q(\{M\})=0,$ then the same conclusion as in Theorem \ref{c2} follows. \end{theorem}
Therefore, under assumptions 1,2,3, whatever $p_0,q$ are, $(p_i)_{i\geq 0}$ converges at least weakly to a unique limit distribution $p^*$ which depends only on $q, M, \beta$.  In particular, if $p_i(\{M\})=0$ for any $i\geq 0$ but $p^*(\{M\})>0$,
then we say the {\it condensation }phenomenon emerges (see \cite{M13} for a study of the phenomenon in Kingman's model and related references therein to other models, e.g., on Bose-Einstein condensation).

In section 5, we will prove that Kingman's model and Lenski model satisfy the three assumptions and therefore limit 
theorems are obtained, especially Theorem \ref{King} for Kingman's model (using Corollary \ref{forking} when it comes to strong convergence). 
 
 \section{Proofs}

\subsection{Proofs of Theorems \ref{c1}, \ref{c11}, Proposition \ref{end} and Corollary \ref{nomass}}

We shall first begin with a technical lemma, which will be used frequently in the sequel.

\begin{lemma}\label{tech}
Let only the assumption 1 be true.  Let $\hat h_0$ and $h$ be two initial type distributions.   The type distribution of generation $i$ is denoted respectively by ${\hat{h}_i}, {h_i}$. The mutant type distribution is $q$ in both cases. If $(\hat{h}_0)_{[0,M)}\prec({{h}_0})_{[0,M)}$$($and hence $\hat h_0(\{M\})\geq h_0(\{M\})$$)$, then for any $i\geq 1$,  $(\hat{h}_i)_{[0,M)}\prec({{h}_i})_{[0,M)}$ $($and $\hat h_i(\{M\})\geq h_i(\{M\})$$)$.
\end{lemma}

\begin{proof}

Assume that for some generation $i\geq 0$, we have $(\hat {h}_i)_{[0,M)}\prec({h}_{i})_{[0,M)}$. By (\ref{dominance}), we have 
\begin{equation}\label{order}s(x,h_i)\geq s(x,\hat{h}_i), \forall x\in\mathcal{M},\end{equation}
together with $(\hat{h}_i)_{[0,M)}\prec(h_{i})_{[0,M)}$: 
\begin{equation}\label{com}\Big(s(x,\hat{h}_i)\hat{h}_i(dx)\Big)_{[0,M)}\prec\Big(s(x,{h}_i){h}_i(dx)\Big)_{[0,M)}.\end{equation}
Then using (\ref{selad}), it is clear that $(\hat {h}_{i+1})_{[0,M)}\prec({h}_{i+1})_{[0,M)}$. So the induction suffices to prove the Lemma.

\end{proof}

\begin{proof}[Proof of Theorem \ref{c1}]
In Lemma \ref{tech}, let $\hat h_i=p_i$ and $h_i=p_{i+1}$ for $i\geq 0$. Then $(\hat h_i)_{i\geq 0}, (h_i)_{i\geq 0}$ satisfy the conditions in Lemma \ref{tech}. Therefore, we obtain $(p_{i})_{[0,M)}\prec (p_{i+1})_{[0,M)}, \forall i\geq 0$. Hence there exists a set function $p^*$ on the measurable sets of $\mathcal{M}$ 
$$\sup_{\text{ any measurable set } A\subset \mathcal{M} }|p^*(A)-p_i(A)|\to 0.$$
So $p^*$ must be a probability measure on $\mathcal{M}$ and $ (p_{i})_{[0,M)}\prec (p_{i+1})_{[0,M)}\prec (p^*)_{[0,M)}.$
\end{proof}

\begin{corollary}\label{limit=}
Let $p_0$ be the Dirac measure at $M$. Then 
\begin{equation}\label{coeff}
s(x,p_i)\text{ increases on $i$ and } s(x,p^*)\geq \sup_{i\geq 0}\{s(x,p_i)\} \text{ for any }x\in \mathcal{M}.  
\end{equation}
And the probability measure $p^*$ satisfies: 
\begin{equation}\label{=}p^*(dx)=(1-\beta)\sup_{i\geq 0}\{s(x,p_i)\}p^*(dx)+\beta q(dx),\end{equation}
and 
 \begin{equation}\label{<}p^*(dx)\prec (1-\beta)s(x,p^*)p^*(dx)+\beta q(dx). \end{equation}
In particular if for any $x\in\mathcal{M}$, $w(x,\cdot)$ is a continuous function on $\P(\mathcal{M})$ equipped with the total variation norm, then (\ref{<}) is in fact an equality.
\end{corollary}

\begin{proof}
By Theorem \ref{c1} and (\ref{dominance}), we obtain (\ref{coeff}).  Recall (\ref{general}):
$$p_i(dx)=(1-\beta)s(x,p_{i-1})p_{i-1}(dx)+\beta q(dx), i\geq 1. $$ 
Due to Theorem \ref{c1}, $p_i$ and $p_{i-1}$ converges strongly to $p^*$ and $s(x,p_{i-1})$ converges increasingly to $\sup_{i\geq 0}\{s(x,p_{i})\}$ for $x\in\mathcal{M}$ fixed as $i$ tends to infinity. So in the limit, using also (\ref{coeff}), we obtain (\ref{=}) and (\ref{<}). The last point is easy to see using some continuity arguments. 

\end{proof}

\begin{proof}[Proof of Proposition \ref{end}]
Let $\hat p_0=\delta_{M}$ and $\hat p_i$ the type distribution at generation $i\geq 1$ under mutant type distribution $q$. In Lemma \ref{tech}, let $\hat h_i=\hat p_i$ and $h_i=p_i$ for any $i\geq 0$. It is easy to see  that Lemma \ref{tech} applies. So we have for any $i\geq 0$
\begin{equation}\label{ppi}(\hat p_i)_{[0,M)}\prec (p_i)_{[0,M)}\text{ and }0\leq p_i(\{M\})\leq \hat p_i(\{M\}).\end{equation}
By Theorem \ref{c1}, $(\hat p_i)_{i\geq 0}$ converges strongly to $p^*$. If $p_i(\{M\})$ converges to $p^*(\{M\})$, using (\ref{ppi}) and the fact that $\int p_i(dx)=\int \hat p_i(dx)=1$, $(p_i)_{i\geq 0}$ converges strongly to $p^*$. 

\end{proof}

\begin{proof}[Proof of Corollary \ref{nomass}]
In the proof of Proposition \ref{end}, since $\hat p_i(\{M\})$ tends to $p^*(\{M\})=0$, using (\ref{ppi}), we also obtain $p_i(\{M\})$ converges to $p^*(\{M\})=0$. The strong convergence follows by Proposition \ref{end}.
\end{proof}

\begin{proof}[Proof of Theorem \ref{c11}]
Note that $p_1(\{M\})>0$. For convenience we assume $p_0(\{M\})>0.$

Let $(\hat p_i)_{i\geq 0}$ be the same as in the proof of Proposition \ref{end} and hence (\ref{ppi}) holds. Let 
\begin{equation}\label{r}R_{i,1}(x)=(1-\beta)^i\prod_{k=0}^{i-1}s(x,p_k), \,\,\, R_{i,2}(x)=\beta\sum_{k=0}^{i-1}(1-\beta)^{k}\prod_{j=0}^{k-1}s(x,p_{i-j})\end{equation}
and
\begin{equation}\label{t}T_{i,1}(x)=(1-\beta)^i\prod_{k=0}^{i-1}s(x,\hat p_k)\1_{x=M}, \,\,\, T_{i,2}(x)=\beta\sum_{k=0}^{i-1}(1-\beta)^{k}\prod_{j=0}^{k-1}s(x,\hat p_{i-j}).\end{equation}
Then it follows from (\ref{selad}) that
\begin{align}\label{pm}p_i(\{M\})&=R_{i,1}(M)p_0(\{M\})+R_{i,2}(M)q(\{M\}),\end{align}
and
\begin{align}\label{pmhat}\hat p_i(\{M\})&=T_{i,1}(M)\hat p_0(\{M\})+T_{i,2}(M)q(\{M\})=T_{i,1}(M)+T_{i,2}(M)q(\{M\}).\end{align}

By Theorem \ref{c1} and assumption 1, $s(x,\hat p_i)\leq s(x,\hat p_{i+1})$ for any $x\in\mathcal{M}$ and $i\geq 0$.  Then we must have $s(M,\hat p_i)\leq 1/(1-\beta)$ for any $i\geq 0$, otherwise $(T_{i,1}(M))_{i\geq 0}$ will tend to infinity which is impossible. Consequently, $(T_{i,1}(M))_{i\geq 0}$ decrease either to $0$ or to a strictly positive value. But in fact, the second option is impossible. Otherwise, as $i$ goes to infinity, $s(M,\hat p_i)$ must converge to $1/(1-\beta)$ which entails that $\liminf_{i\to\infty} T_{i,2}(M)=\infty$ and so $\infty=\liminf_{i\to\infty} T_{i,2}(M)q(\{M\})=\liminf_{i\to\infty}\hat p_i(\{M\})$ (since $q(\{M\})>0$). Therefore $(T_{i,1}(\{M\}))_{i\geq 0}$ converges to $0$. 

By (\ref{ppi}) and assumption 1, we have $s(x, p_i)\geq s(x,\hat p_i)$ for any $x\in \mathcal{M}, i\geq 0$ which entails $R_{i,2}(M)\geq T_{i,2}(M)$,  and $R_{i,1}(M)\leq T_{i,1}(M)/p_0(\{M\})$ to ensure that $p_i(\{M\})<\hat p_i(\{M\})$. Therefore as $i$ goes to infinity, $R_{i,1}(M)$ converges to $0$ and  $p_i(\{M\})$ converge to $p^*(\{M\})$. Strong convergence follows by Proposition \ref{end}. 

\end{proof}

\subsection{Proof of Theorem \ref{c2}}

\begin{proof}[Proof of Theorem \ref{c2}] Let $(\hat p_i)_{i\geq 0}$ be the same as in the proof of Proposition \ref{end} and then (\ref{ppi}) holds. If $p^*(\{M\})=0$, then strong convergence follows due to Corollary \ref{nomass}. So here we assume $p^*(\{M\})>0$. Note that $\hat p_0(\{M\})=1, q(\{M\})=0$. So (\ref{pm}) and (\ref{pmhat}) become

\begin{equation}p_i(\{M\})=R_{i,1}(M)p_0(\{M\}), \,\,\,\hat p_i(\{M\})=T_{i,1}(M).\end{equation}
By the same arguments in the proof of Theorem \ref{c11}, $T_{i,1}(\{M\})$ decreases to $p^*(\{M\})$ (which is strictly positive by assumption) and 
$T_{i,1}(M)\leq R_{i,1}(M)\leq T_{i,1}(M)/p_0(\{M\})$ for any $i\geq 0$.

Let $0\leq a<M$. Due to assumption $2$,  for any $\epsilon>0$ and any $i$, if $p_i([0,a])-\hat p_i([0,a])>\epsilon$ (the difference is always nonnegative due to (\ref{ppi})), then 
\begin{equation}\label{ccf}s(M,p_i)\geq s(M,\hat p_i)c(a,\epsilon),\text{ with } c(a,\epsilon)>1.\end{equation}
However, the above display can hold only for finitely many $i$s. Otherwise, $R_{i,1}(\{M\})$ will tend to infinity.  That implies $\lim_{i\to\infty}p_i([0,a])-\hat p_i([0,a])=0.$ So together with (\ref{ppi}), $(p_i)_{[0,a]}$ converges strongly to $(p^*)_{[0,a]}$. It therefore follows the weak convergence of $(p_i)_{i\geq 0}$ to $p^*.$ 
\end{proof}

\subsection{Proof of Theorem \ref{c3}}
We start with a lemma. 
\begin{lemma}\label{tech1}
Let only the assumption 1 be true. Let $p_0$ and $\hat{p}_0$ be two initial type distributions such that $0<m_{p_0}\leq m_{\hat{p}_0}$ and $(\hat p_0)_{[0,m_{p_0})}\prec (p_0)_{[0,m_{p_0})}$. The type distributions of generation $i$ are denoted respectively by ${p_i}$, ${\hat{p}_i}$. Assume that the mutation type distribution for $p_0$ is $q$ and for $\hat p_0$ is $\hat q$ such that $q=\hat q^{m_{p_0}}$. Then for any $i\geq 1$, 
$(\hat{p}_i)_{[0,m_{p_0})}\prec({{p}_i})_{[0,m_{p_0})}$ and $0_{(m_{p_0},M]}=({{p}_i})_{(m_{p_0},M]}\prec (\hat{p}_i)_{(m_{p_0}, M]}$ where $0_{(m_{p_0},M]}$ is a null measure on ${(m_{p_0},M]}.$
\end{lemma}
\begin{proof}The proof follows the same lines in the proof of Lemma \ref{tech}, using assumption 1.
\end{proof}


Given initial type distribution $p_0$ with mutant type distribution $q$ and another couple $\hat p_0, \hat q$, define a function $\Delta: (p_0,q,\hat p_0,\hat q)\to \{0,1\}$ such that $\Delta(p_0,q,\hat p_0,\hat q)=1$ if and only if these distributions satisfy the assumptions in Lemma \ref{tech1}. 

\begin{proof}[Proof of Theorem \ref{c3}]
Let $0< a<b<M$. Let $h_0=(p_0)^b$. Let $h_i$ be the type distribution of generation $i$ with $h_0$ as initial distribution and $q^b$ as the mutant distribution.  Then $m_{q^b}\leq m_{h_0}=b$ and $h_0(\{b\})+q^b(\{b\})>0$. Due to Theorem \ref{c11} and Theorem \ref{c2}, $(h_i)_{i\geq 0}$ converges strongly to $p^{b,*}$ on any $[0,a]$. 

It is easy to verify that $\Delta(h_0,q^b,p_0,q)=1$ and $\Delta(p_0,q,\delta_{M},q)=1$. Let $(\hat p_i)_{i\geq 0}$ be the same as in (\ref{ppi}). We could apply Lemma \ref{tech1} which gives that 
$$(\hat p_i)_{[0,b)}\prec (p_i)_{[0,b)}\prec (h_i)_{[0,b)}, \forall i\geq 0$$
and in consequence
\begin{equation}\label{prec}(\hat p_i)_{[0,a]}\prec (p_i)_{[0,a]}\prec (h_i)_{[0,a]}.\end{equation}

According to assumption 3, $p^{b,*}$ converge weakly to $p^*$ as $b$ tends to $M$. Then by Portmanteau theorem (\cite{Bi}), 
$\limsup_{b\to M} p^{b,*}([0,a])\leq p^*([0,a])$. Note that as $i$ tends to infinity, $(h_i)_{[0,a]}$ converges strongly to $(p^{b,*})_{[0,a]}$, so as $(\hat p_i)_{[0,a]}$ to $(p^*)_{[0,a]}$. Together with (\ref{prec}), $(p^*)_{[0,a]}\prec(p^{b,*})_{[0,a]}$. Therefore $(p^{b,*})_{[0,a]}$ converges strongly to $(p^*)_{[0,a]}$ as $b$ tends to $M$ from left. Since $b$ can be any number between $a$ and $M$, using again (\ref{prec}), $(p_i)_{[0,a]}$ converges strongly to $(p^*)_{[0,a]}$. Then the proof is completed. \end{proof}

\section{Applications}
\subsection{Kingman's model}
We verify the assumptions $1,2$ and $3$ for Kingman's model. Recall that $M=1$. 

{\it Assumption 1}: It is clear that $u\leq v$ (i.e., $D_u(x)\geq D_v(x), \forall x\geq 0$). Then $\int xu(dx)=\int(1-D_u(x))dx\leq \int(1-D_v(x))dx= \int xv(dx)$ which implies $s(x,u)=\frac{x}{\int xu(dx)}\geq s(x,v)=\frac{x}{\int xv(dx)}$ for any $x\in \mathcal{M}.
$ 

{\it Assumption 2}: We have
\begin{align*}\int xu(dx)=\int(1-D_{u}(x))dx&\leq \int(1-D_{v}(x)-\epsilon\1_{(a,1)}(x))dx=\int x v(dx)-\epsilon(1-a)&\\
&\leq \int x v(dx)(1-\epsilon(1-a))&.\end{align*}
where the second inequality is due to $0< \int xv(dx)\leq 1.$ So
$$s(M,u)=\frac{M}{\int xu(dx)}\geq \frac{M}{\int xv(dx)}\frac{1}{1-\epsilon(1-a)}=s(M,v)\frac{1}{1-\epsilon(1-a)}.$$
One can take $c(a,\epsilon)=\frac{1}{1-\epsilon(1-a)}$ to satisfy assumption $2$.

{\it Assumption 3}:  We need to prove that $p^{a,*}$ converges weakly to $p^*$ as $a$ tends to $1$. We shall derive explicit formulas for $p^{a,*}$ and $p^*$. According to Corollary \ref{limit=}, it is easy to see that  
\begin{equation}p^{a,*}(dx)=(1-\beta)\frac{xp^{a,*}(dx)}{\int xp^{a,*}(dx)}+\beta q^a(dx),\,0\leq a\leq 1.
\end{equation}
That is,
\begin{equation}\label{transform}p^{a,*}(dx)\Big(1-(1-\beta)\frac{x}{\int xp^{a,*}(dx)}\Big)=\beta q^a(dx).\end{equation}
Note that $\int xp^{a,*}(dx)$ determines $p^{a,*}$. By monotonicity, the above equation has a unique solution $p^{a,*}$. 
A simple calculation gives the following:


\begin{lemma}\label{item}
Following the above discussion,  for any $0< a\leq 1,$
\begin{enumerate}
\item If $\int \frac{\beta q^a(dx)}{1-x/a}>1$, then $p^{a,*}(dx)=\frac{\beta s_aq^a(dx)}{s_a-(1-\beta)x}$ where $s_a$ is the unique solution of equation $1=\int \frac{\beta s_aq^a(dx)}{s_a-(1-\beta)x}$.
\item If $\int \frac{\beta q^a(dx)}{1-x/a}\leq 1$, then $p^{a,*}(dx)=\frac{\beta q^a(dx)}{1-x/a}+(1-\int \frac{\beta q^a(dx)}{1-x/a})\delta_a$.
\end{enumerate}
\end{lemma}

\begin{remark}
Notice that if we replace $a$ by $1$, then we obtain the same limit distributions as in Theorem \ref{king}.
\end{remark}

Now we can verify the assumption 3 by Lemma \ref{item}. 
\begin{corollary}[assumption 3]\label{3king}
As $a$ tends to $1$ from left,  $p^{a,*}$ converges weakly to $p^*$.
\end{corollary}

Hence all three assumptions are satisfied by Kingman's model. At least weak convergence of $(p_i)_{i\geq0}$ to $p^*$ is guaranteed. In particular, Corollary \ref{forking} ensures the strong convergence when $\int\frac{q(dx)}{1-x}>\beta^{-1}$ in Theorem \ref{King}. That covers Kingman's result. 

\subsection{Lenski model}
We would like to apply the results from general model (\ref{general}) to Lenski model. As we will show later (Theorem \ref{lc}), given any $p_0,q$, $(p_i)_{i\geq 0}$ converges at least in weak sense to a unique limit distribution which depends only on $q, M, \beta$. One can imagine that, at some day, the type distribution is $p_0,$ and then next day a positive proportion is replaced by a fitter mutant population with distribution $q$. Assume this procedure can iterate for many days. Then we shall observe that as $i$ tends to infinity that the population gets fitter and fitter but in a deceleration speed to achieve some limit. Note that at the early stage (around $10,000$ generations), a hyperbolic curve, which has a bound, fits to the data. Our general model also has a bound, in the spirit of hyperbolic curve. As time moves on, fitter mutant distribution pattern will appear and the previous bound will be broken up. We shall then in the long term see that the type values increase unboundedly and also sublinearly, since fitter mutant distribution patterns are rarer. So our model provides another point of view to explain the increasing trend of the type values. 

We shall prove that assumptions $1$ to $3$ are satisfied by Lenski model and find the unique limit distribution.

{\it Assumptions 1 and 2:} One just needs to follow the same procedure as for Kingman's model by using distribution functions.  \\
{\it Assumption 3:} The same as for Kingman's model, we give the expression of $p^{a,*}$ for any $0<a\leq M$. Similarly, we have 
\begin{equation}\label{lenski}p^{a,*}(dx)=(1-\beta)\frac{e^{t_{p^{a,*}}x}p^{a,*}(dx)}{\gamma}+\beta q^a(dx)
\end{equation}
which has a unique solution $p^{a,*}.$

In the same spirit of Lemma \ref{item} and Corollary \ref{3king}, we have the corresponding results in Lenski model.
\begin{lemma}\label{item2}
Following the above discussion,  for any $0<a\leq M$
\begin{enumerate}
\item If $\int\frac{\beta q^a(dx)}{1-(\frac{1-\beta}{\gamma})^{1-x/a}}>1$, then $p^{a,*}(dx)=\frac{\beta q^a(dx)}{1-\frac{1-\beta}{\gamma}e^{s_ax}}$, where $s_a$ is the unique solution of equation $1=\int \frac{\beta q^a(dx)}{1-\frac{1-\beta}{\gamma}e^{s_ax}}.$
\item If $\int\frac{\beta q^a(dx)}{1-(\frac{1-\beta}{\gamma})^{1-x/a}}\leq 1$, then $p^{a,*}(dx)=\frac{\beta q^a(dx)}{1-(\frac{1-\beta}{\gamma})^{1-x/a}}+(1-\int\frac{\beta q^a(dx)}{1-(\frac{1-\beta}{\gamma})^{1-x/a}})\delta_a$.
\end{enumerate}
\end{lemma}

\begin{corollary}[assumption 3]
In the Lenski model, $p^{a,*}$ converges weakly to $p^*$ as $a$ tends to $M$. 
\end{corollary}

So all three assumptions are satisfied by Lenski model, which guarantees the weak convergence of $(p_i)_{i\geq 0}$ to $p^*$. Together with Corollary \ref{forking} and replacing $a$ by $M$ in Lemma \ref{item2}, we obtain the main result as follows: 
\begin{theorem}\label{lc}
\begin{enumerate}
\item If $\int\frac{\beta q(dx)}{1-(\frac{1-\beta}{\gamma})^{1-x/M}}>1$, then $(p_i)_{i\geq 0}$ converges strongly to $p^*(dx)=\frac{\beta q(dx)}{1-\frac{1-\beta}{\gamma}e^{s_ax}}$, where $s_a$ is the unique solution of equation $1=\int \frac{\beta q(dx)}{1-\frac{1-\beta}{\gamma}e^{s_ax}}.$
\item If $\int\frac{\beta q(dx)}{1-(\frac{1-\beta}{\gamma})^{1-x/M}}\leq 1$, then $(p_i)_{i\geq 0}$ converges at least weakly to $p^*(dx)=\frac{\beta q(dx)}{1-(\frac{1-\beta}{\gamma})^{1-x/M}}+(1-\int\frac{\beta q(dx)}{1-(\frac{1-\beta}{\gamma})^{1-x/M}})\delta_M$.
\end{enumerate}
\end{theorem}

\section{Discussions}
The main generalization in our model is the dependence of the fitness function on the current type distribution. 
Biologically, it means that the fitness of an individual
depends closely on his contemporaries. In other words, macroscopic epistasis enter explicitly into effect. 
That allows us to design different fitness functions according to our purposes,  such as in the case of Lenski experiment. 
Moreover the assumptions made on the fitness function biologically make sense and are easy to verify, which reduces the modelling difficulties of other potential 
biological experiments. In particular the generalization is also a simplification, because the tools employed in our approach such as the stochastic dominance and some convergence notions of (probability) measures are simple but good enough to reveal the 
nature of the asymptotic behavior of type distributions. 

Notice that we always assume $M<\infty. $ However the iteration (\ref{general}) can be well defined with carefully chosen $p_0,q, w$ when $M=\infty$. Kingman \cite{K78}  studied an example in the infinite case.
However, no general dynamics exist and the only thing one can expect is that a certain proportion of the mass of $p_i$ moves to the right as far as possible as $i$ increases.

B\"urger \cite{B89,B98} considered fitness functions which does not depend on current type distribution and gathered some assumptions on the fitness function to make it possible 
to work on $M\leq \infty$.  It turns out that, under appropriate conditions, $(p_i)_{i\geq 0}$ converges strongly 
 to a unique limit distribution. However this generalization does not cover 
 neither Kingman's \textit{Case} $1$ nor the situations with condensation phenomenon. 

The Lenski experiment has attracted much attention recently, see for instance a survey by Ycart et al \cite{YHGS} and references therein. 
Another paper \cite{GKWY} builds a very stochastic approach to analyze the experiment and 
our model can be considered as its deterministic counterpart. There are advantages to stay stochastic and 
also to stay deterministic. In \cite{GKWY}, the authors considered a population with finite population $N$. Each time 
a mutation can only affect one individual. There arises the interesting question of studying the random evolution of mutant population size compared to the total population size. Another advantage is that one can ignore the boundedness concern of the type values if we assume each mutation brings very slight changes to the type value. To stay deterministic, one can quickly
have a clear idea on how the two forces, selection and mutation, compete in the game. In particular, the macroscopic epistasis is 
designable and the deterministic general model can cover various other evolutionary models. In view of this comparison, a future research topic would be considering how to merge the stochastic and deterministic models in order to provide a clear global image and also allow dynamical randomnesses.

\section{Acknowledgements} I thank Adrian Gonzalez-Casanova, Noemi Kurt and Anton Wakolbinger for stimulating discussions on Lenski experiment. Thank especially Noemi for her very helpful comments on a draft version.

\end{document}